\documentclass{amsart}

\usepackage{amscd,amssymb}

\newtheorem{theorem}{Theorem}[section]
\newtheorem{lemma}[theorem]{Lemma}
\newtheorem{proposition}[theorem]{Proposition}
\newtheorem{corollary}[theorem]{Corollary}

\newtheorem{conjecture}[theorem]{Conjecture}

\theoremstyle{definition}

\newtheorem{example}[theorem]{Example}

\newtheorem{remark}[theorem]{Remark}

\newtheorem{algorithm}[theorem]{Algorithm}

\DeclareMathOperator{\GL}{GL}

\DeclareMathOperator{\tr}{tr}

\newcommand{\C}{{\mathbb{C}}}
\newcommand{\op}{\mathrm{op}}

\newcommand{\KXd}{K\langle X_d\rangle}

\newcommand{\CXdCd}{{\mathbb C}\langle X_d\rangle^{C_d}}

\begin{document}

\title[On cyclic invariants of the free associative algebra]
{On cyclic invariants of the free associative algebra}

\author[Silvia Boumova, Vesselin Drensky]
{Silvia Boumova$^{1}$, Vesselin Drensky$^{2}$ }
\address{$^{1}$Faculty of Mathematics and Informatics,
Sofia University, 1164 Sofia, Bulgaria and Institute of Mathematics and Informatics,
Bulgarian Academy of Sciences,
1113 Sofia, Bulgaria\\}
\email{boumova@fmi.uni-sofia.bg}
\address{$^{2}$Institute of Mathematics and Informatics,
Bulgarian Academy of Sciences,
1113 Sofia, Bulgaria}
\email{drensky@math.bas.bg}

\subjclass[2020]{13A50, 15A72, 16S10, 16W22}
\keywords{classical invariant theory, noncommutative invariant theory, cyclic invariants, action of Kharchenko-Koryukin}

\begin{abstract}

Let $K\langle X_d\rangle$ be the free associative algebra of rank $d \geq 2$ over a field $K$.
Lane in 1976 and Kharchenko in 1978 proved that the algebra of invariants $K\langle X_d\rangle^G$ is free for any subgroup $G \leq \text{GL}_d(K)$ and any field $K$.
Later, Kharchenko introduced an additional action of the symmetric group $\text{Sym}(n)$ on the homogeneous component of degree $n$ of $K\langle X_d\rangle$,
given by permuting the positions of the variables. This equips $K\langle X_d\rangle$ with the structure of a $(K\langle X_d\rangle,\circ)$-$S$-algebra.
Then Koryukin showed that the algebra of invariants $K\langle X_d\rangle^G$ is finitely generated for every reductive group $G$ with respect to this action.

In our paper we study the algebra $K\langle x_1,\ldots,x_d\rangle^{C_d}$ of invariants of the cyclic group $C_d$, $d\geq 2$, where $K$ is an arbitrary field of characteristic 0.
We compute the Hilbert series of $K\langle x_1,\ldots,x_d \rangle^{C_d}$. When $K=\mathbb C$ we find a vector space basis of ${\mathbb C}\langle x_1,\ldots,x_d \rangle^{C_d}$
and explicitly describe the generators of ${\mathbb C}\langle x_1,\ldots,x_d \rangle^{C_d}$ as a free algebra.
Moreover, we describe a finite generating set for the $S$-algebra $({\mathbb C}\langle x_1,\ldots,x_d \rangle^{C_d},\circ)$.
We also transfer the results for $K=\mathbb C$ to the case of an arbitrary field of characteristic 0 for the $S$-algebra $(K\langle x_1,x_2,x_3 \rangle^{C_3},\circ)$
and find a minimal generating set for it as an $S$-algebra.
\end{abstract}

\maketitle

\section{\textbf{Introduction}}
Let $K$ be an arbitrary field of characteristic zero and let $K[X_d] = K[x_1,\ldots,x_d]$  be the $d$-generated polynomial algebra (the free commutative associative unital algebra of rank $d>1$) over $K$.
Classical invariant theory studies the algebra of invariants $K[X_d]^G$ which consists of all polynomials invariant under the action of a subgroup $G$ of the general linear group $\text{GL}_d(K)$.

The originals of the theory can be found in the work of Lagrange in the 1770s and Gauss in 1801.
In his {\it M\'echanique analitique} \cite{La} Lagrange considered the problem of the transformation of homogeneous polynomials by linear substitutions of the variables.
In {\it Disquititiones Arithmeticae} \cite{Ga} Gauss studied the representation of integers by quadratic binary forms and used the discriminant to distinguish nonequivalent forms.

But it is accepted that the real invariant theory began in the 1840s with works by George Boole in England and by Otto Hesse in Germany.
See Corry \cite{Co} and Wolfson \cite{Wol} for details and references.
For the history of the early invariant theory see also the survey by Meyer \cite{Mey} from 1892, the books by Klein \cite{Kl} and Weyl \cite[the beginning of Chapter II]{Wey} and
the papers by Parshall \cite{Par} and Rota \cite{Ro}. See also the references in our paper \cite{BDDK}.
In two papers published in 1841 and 1842 Boole continued the work of Lagrange on linear transformations of homogeneous polynomials.
He established the principles of invariance  and introduced a series of invariant expressions.
The ideas of Boole were further developed by Cayley in a series of papers started in 1845, who was the first to use the word ``invariant''.
Later, in 1853, Sylvester introduced the term ``syzygies'' for the relations between invariants.
The approach of Hesse was inspired by geometry: the study of critical points in plane curves with the help of the Hessian determinant.

Later, the development of the theory continued in the work of Salmon in Great Britain, Aronhold, Clebsch and  Gordan
in Germany, Brioschi and Cremona in Italy, Hermit in France.
According to Weyl the main line of the research in this period was ``the development of formal processes and the actual computation of invariants''.

The next step of the development of invariant theory was in the work of Hilbert in 1890-1892
who developed general methods to solve into affirmative many of the main problems of the theory.
His nonconstructive proofs were accepted sceptically by many of his contemporary mathematicians.
Gordan (known as ``K\"onig der Invariantentheorie'' (the king of invariant theory)) commented that
{\it ``Das ist nicht Mathematik, das ist Theologie.''} (``This is not mathematics, this is theology.''
It is not clear if Gordan really said this since the earliest reference to it is 25 years after the events and after his death,
and nor is it clear whether the quote was intended as criticism, or praise, or a subtle joke.)

A special role is played by the  14th problem of Hilbert \cite{Hi2} from his famous talk at the International Congress of Mathematicians in 1900
which was motivated by the problem whether algebras of invariants are always finitely generated.
In 1915, Emmy Noether \cite{No1} proved that if $G$ is finite, then the algebra $K[X_d]^G$ is finitely generated. (In 1926 she extended this to the case of fields of any characteristic \cite{No2}.)
For reductive linear groups, finite generation also follows from classical invariant theory. However, in general the answer to the 14th problem of Hilbert is negative.
In 1959 Nagata \cite{Na} gave a counterexample showing that the algebras $K[X_d]^G$ are not finitely generated for a class of infinite linear groups.

We now turn to the noncommutative case. Here, the commutative polynomial algebra $K[X_d]$ is replaced by its natural noncommutative counterpart:
the free associative algebra $K\langle X_d\rangle$ -- the algebra of noncommuting polynomials in $d$ variables that satisfies the same fundamental universal property:
For any associative algebra $R$ every map $X_d\to R$ can be uniquely extended to a homomorphism $K\langle X_d\rangle \to R$.

The first paper on invariant theory of groups acting on the free associative algebra was published in 1936 by Margarete Wolf \cite{Wo}. She described the symmetric polynomials in
$K\langle X_d\rangle$, i.e. the elements of the algebra of invariants $K\langle X_d\rangle^{\text{\rm Sym}_d}$, proved that $K\langle X_d\rangle^{\text{\rm Sym}_d}$ is a free algebra which is not finitely generated
and found a free generating set of this algebra.

The noncommutative analogue of the theorem of Noether on finite generation turns out to be far more restrictive.
For finite subgroups $G \leq \text{GL}_d(K)$, results of Dicks and Formanek \cite{DiFo} and Kharchenko \cite{Kh2} in the 1980s show that the algebra $K\langle X_d\rangle^G$ is finitely generated
if and only if $G$ acts by scalar multiplication on the vector space spanned by the variables $X_d$. This characterization was later generalized to infinite groups by Koryukin \cite{Kor}.

Kharchenko \cite{Kh2} 
introduced an action of the symmetric group $\text{\rm Sym}_n$ on the  homogeneous component of degree $n$ in $\KXd$, defined by permuting the positions of the variables.
In this way $K\langle X_d\rangle=(\KXd,\circ)$ becomes an $S$-algebra. Kharchenko asked whether the algebra of $G$-invariants $(K\langle X_d\rangle^G,\circ)$ is finitely generated as an $S$-algebra
for all finite groups $G$. In 1984, Koryukin \cite{Kor} proved that $(K\langle X_d\rangle^G,\circ)$ is finitely generated for every reductive group $G$.

A central problem in both commutative and noncommutative invariant theory is to describe the invariants of significant groups.
In the present paper we study the algebra ${\mathbb C}\langle X_d\rangle^{C_d}$ of invariants under the action of the cyclic group $C_d$ on the free associative algebra ${\mathbb C}\langle X_d\rangle$.
We compute its Hilbert series and construct a homogeneous set of
generators of ${\mathbb C}\langle X_d\rangle^{C_d}$ as a free algebra. Finally, we find a finite generating set of ${\mathbb C}\langle X_3\rangle^{C_3}$ as an $S$-algebra.

Related problems concerning noncommutative symmetric polynomials, i.e. for the algebra of invariants $\C\langle X_d \rangle^{\text{\rm Sym}_d}$ were considered in \cite{BDDK}
and in \cite{BDF} for the invariants ${\mathbb C}\langle x_1,x_2\rangle^{D_{2n}}$ of the dihedral group $D_{2n}$.

\section{\textbf{Preliminaries}}
The general linear group $\GL_d(\C)$ acts naturally on the $d$-dimensional complex vector space $V_d$ with basis $\{v_1,\ldots,v_d\}$.
The coordinate functions $x_i:V_d\to {\mathbb C}$ are defined by
\[
x_i(\xi_1v_1+\cdots+\xi_dv_d)=\xi_i,\quad \xi_1,\ldots,\xi_d\in \C, \quad i=1,\ldots,d.
\]
This induces an action of $\GL_d(\C)$ on the algebra of polynomial functions $\C[X_d]={\mathbb C}[x_1,\ldots,x_d]$ via
\[
g(f):v\to f\big(g^{-1}(v)\big),\quad g\in \GL_d(\C),\, f(X_d)\in \C[X_d],\, v\in V_d.
\]
For a subgroup $G$ of $\GL_d(\C)$, the algebra of $G$-invariants consists of all polynomials fixed under this action of $G$, i.e.
\[
\C[X_d]^G=\big\{f\in \C[X_d]\mid g(f)=f\text{ for all }g\in G\big\}.
\]
For our purposes it is more convenient to assume that $\GL_d(\C)$ acts canonically on the vector space $KX_d$ with basis $X_d$ and extend its action diagonally to $\C[X_d]$ by
\[
g\big(f(x_1,\ldots,x_d)\big)=f\big(g(x_1),\ldots,g(x_d)\big),\quad g\in \GL_d(\C), f\in \C[X_d].
\]
The action of $\GL_d(\C)$ on the polynomial functions $\C[X_d]$ coincides
with the diagonal action of its opposite group $\GL_d^\op(\C)$ induced by the canonical action of $\GL_d^\op(\C)$ on the vector space $KX_d$.
Since the map  $g\to g^{-1}$ is an isomorphism of $\GL_d(\C)$ and $\GL_d^\op(\C)$, the invariant algebras obtained from these two actions are the same.

\subsection{Foundational Theorems}
The foundations of commutative invariant theory for finite groups $G$ rest on several classical theorems.

\begin{theorem}\label{Endlichkeitssatz} {\rm (Endlichkeitssatz of Emmy Noether \cite{No1})}
If $G$ is a finite subgroup $G$ of $\GL_d(K)$ and the field $K$ of characteristic zero, then the algebra of invariants $K[X_d]^G$ is finitely generated.
Moreover, its homogeneous generators can be chosen so that their degrees are bounded by the order of $G$.
\end{theorem}

\begin{theorem}\label{Basissatz} {\rm (Basisssatz of Hilbert \cite{Hi1})}
Over an arbitrary field $K$ of any characteristic every ideal of $K[X_d]$ is finitely generated.
\end{theorem}

As a consequence of Theorems \ref{Endlichkeitssatz} and \ref{Basissatz} one immediately obtains that if $G$ is a finite group
and the set of polynomials $\{f_1(X_d),\ldots,f_n(X_d)\}$ generates $K[X_d]^G$,
then there are finitely many polynomials (syzygies) $w_1(Y_n),\ldots,w_m(Y_n)$ such that the factor algebra $K[Y_n]/(w_1(Y_n),\ldots,w_m(Y_n))$
of the polynomial algebra $K[Y_n]$ modulo the ideal generated by $w_1(Y_n),\ldots,w_m(Y_n)$ is isomorphic to $K[X_d]^G$.

\begin{theorem}\label{Chevalley-Shephard-Todd} {\rm (Chevalley \cite{Che} and Shephard and Todd \cite{ShT})}
If the field $K$ is of characteristic zero and $G$ is a finite subgroup of $\GL_d(K)$, then the algebra of invariants $K[X_d]^G$ is isomorphic to a polynomial algebra,
i.e. has a system of algebraically independent generators if and only if $G$ is generated by pseudo-reflections
(matrices of finite order conjugate to a diagonal matrix of the form $\mathrm{diag}(1,\dots,1,\xi)$, where $\xi\neq  1$ is a root of unity).
\end{theorem}

The symmetric and the dihedral groups are among the well known groups which satisfy the conditions of this theorem.

The algebra $K[X_d]^G$ decomposes into its homogeneous components:
\[
K[X_d]^G=K\oplus \left(K[X_d]^G\right)^{(1)} \oplus \left(K[X_d]^G\right)^{(2)} \oplus \cdots.
\]
Then the Hilbert (or Poincar\'e) series of $K[X_d]^G$ is the formal power series
\[
H\big(K[X_d]^G,t\big)=\sum_{n\geq 0}\dim \left(K[X_d]^G\right)^{(n)}t^n.
\]
According to the Hilbert-Serre theorem, this series is rational with a denominator of a very special form:
\[
H\big(K[X_d]^G,t\big) = p(t) \prod_{i=1}^m\frac{1}{1-t^{a_i}},\quad p(t)\in{\mathbb Z}[t].
\]
The explicit form of the Hilbert series of $K[X_d]^G$ is given by Molien in 1897.

\begin{theorem}\label{Molien formula} {\rm (Molien \cite{Mo})}
Over a field of characteristic zero the Hilbert series of the algebra of invariants $K[X_d]^G$ of a finite group $G$ is
\[
H\big(K[X_d]^G,t\big)=\frac{1}{\vert G\vert}\sum_{g\in G}\frac{1}{\det\big(1-gt\big)}.
\]
\end{theorem}

\subsection{Invariants of the cyclic group}
Now we summarize the facts for the invariants of the cyclic group $C_d$ of order $d$ needed for for our project.
We assume that $G=C_d$ is generated by $\rho$ which shifts cyclically the variables $x_1,\ldots,x_d$:
\[
\rho(x_1)=x_2,\rho(x_2)=x_3,\ldots,\rho(x_{d-1})=x_d,\rho(x_d)=x_1.
\]
Then for any field $K$ of characteristic zero $K[X_d]^{C_d}$ is spanned by all
\[
\frac{1}{d}\sum_{k=0}^{d-1}\rho^k(x_1^{n_1}\cdots x_d^{n_d}),\; n_i\geq 0.
\]
In the sequel we shall work over $K=\mathbb C$.
We shall change the basis $X_d$ of ${\mathbb C}X_d$ to
\[
y_k=x_1+\varepsilon^{-k}x_2+\varepsilon^{-2k}x_3+\cdots \varepsilon^{-k(d-1)}x_d,k=0,1,\ldots,d-1,
\]
where $\displaystyle \varepsilon=\cos\left(\frac{2\pi}{d}\right)+i\sin\left(\frac{2\pi}{d}\right)$ is a primitive $d$-root of 1.
Then
\[
\rho(y_k)=\varepsilon^ky_k, k=0,1,\ldots,d-1,
\]
\[
\rho(y_0^{n_0}y_1^{n_1}\cdots y_{d-1}^{n_{d-1}})=\varepsilon^{n_1+2n_2+\cdots+(d-1)n_{d-1}}y_0^{n_0}y_1^{n_1}\cdots y_{d-1}^{n_{d-1}}
\]
and ${\mathbb C}[X_d]^{C_d}={\mathbb C}[y_0,y_1,\ldots,y_{d-1}]^{C_d}={\mathbb C}[Y_d]^{C_d}$ is spanned by the monomials $y_0^{n_0}y_1^{n_1}\cdots y_{d-1}^{n_{d-1}}$ such that
$n_1+2n_2+\cdots+(d-1)n_{d-1}\equiv 0\text{ (mod }d)$.

The following statement is well known and follows easily from Theorem \ref{Endlichkeitssatz} of Emmy Noether.

\begin{proposition}\label{commutative case}
The algebra of invariants ${\mathbb C}[Y_d]^{C_d}$ is generated by $y_0$ and the monomials
$u=y_1^{n_1}\cdots y_{d-1}^{n_{d-1}}$ such that
\[
1\leq n_1+\cdots+n_{d-1}\leq d,n_1+2n_2+\cdots+(d-1)n_{d-1}\equiv 0 \text{ \rm (mod }d)
\]
and which cannot be presented as products of two or more polynomials of this kind.
\end{proposition}

In other words, we consider the multiplicative semigroup $[Y_d]=[y_0,y_1,\ldots,y_{d-1}]$ of all monomials in $d$ variables. Then the algebra ${\mathbb C}[Y_d]^{C_d}$ of invariants of $C_d$
is spanned by the subsemigroup of all monomials
\[
y_0^{n_0}y_1^{n_1}\cdots y_{d-1}^{n_{d-1}}\text{ with }n_1+2n_2+\cdots+(d-1)n_{d-1}\equiv 0 \text{ \rm (mod }d),
\]
and is generated by the monomials which generate this subsemigroup.

\begin{example}\label{invariants with diagonal action}
(i) $d=3$. All monomials $y_1^{n_1}y_2^{n_2}$, $n_1+n_2\leq 3$, $n_1+2n_2 \equiv 0 \text{ \rm (mod }3)$ are
$y_1y_2,y_1^3,y_2^3$ and ${\mathbb C}[y_0,y_1,y_2]^{C_3}$ is generated by $y_0,y_1y_2,y_1^3,y_2^3$.

(ii) $d=4$. The number of monomials $y_1^{n_1}y_2^{n_2}y_3^{n_3}$ of degree $2,3$ and $4$ and
such that $n_1+2n_2+3n_3 \equiv 0 \text{ \rm (mod }4)$ is equal, respectively, to $2, 2$ and $5$.
The monomials $y_1y_3,y_2^2,y_1^2y_2,y_2y_3^2,y_1^4,y_3^4$ together with $y_0$ generate
${\mathbb C}[y_0,y_1,y_2,y_3]^{C_4}$. The other $3$ monomials of degree $4$
\[
y_1^2y_3^2=(y_1y_3)^2,y_1y_2^2y_3=(y_1y_3)(y_2^2), y_2^4=(y_2^2)^2
\]
are products of the first two monomials $y_1y_3,y_2^2$.

(iii) $d=5$. It is easy to see that ${\mathbb C}[y_0,y_1,y_2,y_3,y_4]^{C_5}$ is generated by
\[
\deg=1:y_0;\quad\deg=2:y_1y_4,y_2y_3;\quad\deg=3:y_1^2y_3,y_1y_2^2;
\]
\[
\deg=4:y_1^3y_2,y_1y_3^3,y_2^3y_4,y_3y_4^3;\quad\deg=5:y_1^5,y_2^5,y_3^5,y_4^5.
\]
\end{example}

If $K$ is an arbitrary field of characteristic 0 and we consider the action of $C_d$ on $K[X_d]$
by cyclical shifting of the variables it is natural to start with the elementary symmetric polynomials
\[
e_n(x_1,\ldots,x_d)=\sum_{i_1<\cdots<i_n} x_{i_1}\cdots x_{i_n},\;n=1,\ldots,d,
\]
and then to extend them to a system of generators of $K[X_d]^{C_d}$ of the form
\[
x_1^{n_1}x_2^{n_2}\cdots x_{d-1}^{n_{d-1}}x_d^{n_d}+x_2^{n_1}x_3^{n_2}\cdots x_d^{n_{d-1}}x_1^{n_d}+\cdots+x_d^{n_1}x_1^{n_2}\cdots x_{d-2}^{n_{d-1}}x_{d-1}^{n_{d-1}}.
\]
Alternatively, instead to start with the elementary symmetric polynomials, we can search for a generating set which consists only of cyclic sums of monomials.

The case $d=2$ is trivial, because the algebra $K[X_2]^{C_2}$ coincides with the algebra of symmetric polynomials in two variables which is generated by $e_1,e_2$.
In the next example we give systems of generators of $K[X_d]^{C_d}$ for $d=3$ and 4.

\begin{example}\label{three and four commuting variables}
(i) $d=3$. By Example \ref{invariants with diagonal action} (i) we have to add one polynomial of degree $3$ to the three symmetric polynomials $e_1,e_2,e_3$.
There are two possibilities:
\[
u_1=x_1^2x_2+x_2^2x_3+x_3^2x_1,\quad u_2=x_1x_2^2+x_2x_3^2+x_3x_1^2.
\]
Since $u_1+u_2$ is a symmetric polynomial, $u_2$ can be expressed in terms of $e_1,e_2,e_3$ and $u_1$. Hence
$K[X_3]^{C_3}$ is generated by $e_1,e_2,e_3,u_1$. Pay attention that here $e_1, e_2, e_3$ are cyclic sums of $x_1$, $x_1x_2$ and $x_1x_2x_3$, respectively.

(ii) $d=4$. By Example \ref{invariants with diagonal action} (ii) the algebra $K[X_4]^{C_4}$ is generated by
$e_1,e_2,e_3,e_4$ and three polynomials of degree $2$, $3$ and $4$. Easy computations show that these three polynomials are
\[
u=x_1x_2+x_2x_3+x_3x_4+x_4x_1,
\]
\[
v=x_1^2x_2+x_2^2x_3+x_3^2x_4+x_4^2x_1,
\]
\[
w=x_1^3x_2+x_2^3x_3+x_3^3x_4+x_4^3x_1.
\]
The other cyclic sums of degree $\leq 4$ can be expressed in terms of $e_1,e_2,e_3,e_4,u,v,w$:
\[
x_1x_3+x_2x_4=e_2-u,
\]
\[
x_1^2x_3+x_2^2x_4+x_3^2x_1+x_4^2x_2=e_1e_2-e_3-e_1u,
\]
\[
x_1x_2^2+x_2x_3^2+x_3x_4^2+x_4x_1^2=-2e_3+e_1u-v,
\]
\[
x_1^3x_3+x_2^3x_4+x_3^3x_1+x_4^3x_2=e_1^2e_2-e_1e_3-2e_2^2+4e_4-ue_1^2+3ue_2-u^2,
\]
\[
x_1x_2^3+x_2x_3^3+x_3x_4^3+x_4x_1^3=ue_1^2-3ue_2+u^2-w,
\]
\[
x_1^2x_2^2+x_2^2x_3^2+x_3^2x_4^2+x_4^2x_1^2=-2e_1e_3+4e_4+2ue_2-u^2,
\]
\[
x_1^2x_3^2+x_2^2x_4^2=e_2^2-2e_4-2ue_2+u^2.
\]
It is easy to see that $K[X_4]^{C_4}$ is generated also by the cyclic sums of
\[
x_1,x_1x_2,x_1x_3,x_1^2x_2,x_1^2x_3,x_1^3x_2,x_1^3x_3.
\]
\end{example}

\subsection{Noncommutative Invariant Theory}
Now we go to invariant theory of groups acting on the free associative algebra $K\langle X_d\rangle$.
The general linear group $\GL_d(K) = \GL(KX_d)$ acts naturally on the vector space $KX_d$ with basis $X_d$. Via algebra homomorphisms, this action extends to $\KXd$:
\[
g(f(x_1,\cdots, x_d )) = f(g(x_1),\cdots, g(x_d)),\, g \in \GL_d , f(x_1,\cdots, x_d ) \in \KXd.	
\]
Clearly, for every subgroup $G \leq \GL_d(K)$ the algebra of $G$-invariants $\KXd^G$ is the set of all polynomials unchanged by this action of the elements of $G$.

\subsection{Fundamental Results on Finite Generation}
The following results illustrate the strong divergence between commutative and noncommutative invariant algebras with respect to finite generation.
They were found independently by Dicks and Formanek \cite{DiFo} and Kharchenko \cite{Kh2} for finite groups and extended by Koryukin \cite{Kor} to arbitrary groups.

\begin{theorem}\label{noncommutative finite generation}{\rm(\cite{Kor})}
Let $G \leq \GL_d(K)$ over any field $K$, and let $KZ_m$ be the minimal subspace of $KX_d$ with $\KXd^G \subseteq K\langle Z_m\rangle$.
Then the algebra $\KXd^G$ is finitely generated if and only if $G$ acts on $KZ_d$ by scalar multiplication.
\end{theorem}

\begin{corollary}{(\cite{DiFo, Kh2})}
For finite $G \leq \GL_d(K)$, the algebra of invariants $\KXd^G$ is finitely generated if and only if $G$ is cyclic and consists of scalar matrices.
\end{corollary}

\begin{corollary}{\rm(\cite{Kor})}
If $G$ acts irreducibly on $KX_d$, then $\KXd^G$ is either trivial or not finitely generated.
\end{corollary}

The analog of the Chevalley-Shephard-Todd theorem for the free associative algebra $K\langle X_d\rangle$ also takes a substantially different form.

\begin{theorem}\label{free generation of noncommutative invariants}
{\rm (i) (Lane \cite{Lan}, Kharchenko \cite{Kh1})} For any $G \leq \GL_d(K)$ and any field $K$, the algebra $\KXd^G$ is free.
		
{\rm (ii) (Kharchenko \cite{Kh1})} For finite groups $G$, there exists a Galois correspondence
between free subalgebras of $\KXd$ containing $\KXd^G$ and subgroups of $G$:
a subalgebra $F$ with $\KXd^G \subseteq F$ is free if and only if $F = \KXd^H$ for some $H \leq G$.
\end{theorem}

The analogue of the formula of Molien for the Hilbert series of the algebra of invariants is expressed in the following elegant way.

\begin{theorem} [\label{noncommutatiev Molien formula}{Dicks and Formanek~{\cite{DiFo}}}]
For a finite group $G \leq \GL_d(K)$ with $\mathrm{char}(K) = 0$,
\[
H\big(K\langle X_d\rangle^G,t\big)=\frac{1}{\vert G\vert}\sum_{g\in G}\frac{1}{1-\tr(g)t}.
\]
\end{theorem}

For thorough expositions on this topic, see the surveys by Formanek \cite{Fo} and Drensky \cite{Dr1} and the paper \cite{BDDK}.

In the next section we shall use the following well known fact.

\begin{lemma}\label{generating function}
Let $Z=Z_1\cup Z_2\cup \cdots$ be a set of homogeneous variables such that $\deg(Z_n)=n$,
and $\vert Z_n\vert=g_n$, $n=1,2,\ldots$. Then the Hilbert series of the free algebra $K\langle Z\rangle$
with respect to the grading induced by the grading of the set $Z$ is
\[
H(K\langle Z\rangle,t)=\frac{1}{1-g(t)},\text{ where }g(t)=\sum_{n\geq 1}g_nt^n.
\]
\end{lemma}

\section{\textbf{The algebra $\CXdCd$ of invariants of the cyclic group}}
As in the commutative case we assume that $G=C_d$ is generated by $\rho$ which shifts cyclically the variables $x_1,\ldots,x_d$.
The next theorem gives the Hilbert series and information on the free generators of the algebra of invariants $K\langle X_d\rangle^{C_d}$.

\begin{theorem}\label{Hilbert series and free generators}
{\rm (i)} For any field $K$ of characteristic zero the Hilbert series of the algebra of invariants $K\langle X_d\rangle^{C_d}$ is
\[
H(K\langle X_d\rangle^{C_d},t)=\frac{1-(d-1)t}{1-dt}\text{ and }\dim\left((K\langle X_d\rangle^{C_d})^{(n)}\right)=d^{n-1},\; n=1,2,\ldots.
\]

{\rm (ii)} The generating function of the set of homogeneous free generators of the algebra $K\langle X_d\rangle^{C_d}$ is
\[
g(t)=\frac{t}{1-(d-1)t},
\]
i.e. for each $n\geq 1$ the system of homogeneous free generators has $(d-1)^{n-1}$ elements of degree $n$.
\end{theorem}

\begin{proof}
(i) The matrix of the generator $\rho$ of $C_d$ with respect to the basis $X_d$ is
\[
\begin{pmatrix}
0& 0 &\dots &0&1\\
1& 0 &\dots &0&0\\
0& 1 &\dots &0&0\\
\vdots&\vdots&\ddots&\vdots&\vdots\\
0& 0 &\dots &1&0\\
\end{pmatrix}.
\]
Hence $\text{tr}(\rho^k)=0$ for $k=1,2,\ldots,d-1$, and $\text{tr}(\rho^0)=\text{tr}(I_d)=d$. Hence Theorem \ref{noncommutatiev Molien formula} gives
\[
H(K\langle X_d\rangle^{C_d},t)=\frac{1}{d}\sum_{k=0}^{d-1}\frac{1}{1-\text{tr}(\rho^k)t}=\frac{1}{d}\left(\frac{1}{1-dt}+(d-1)\right)
\]
\[
=\frac{1-(d-1)t}{1-dt}=1+\frac{t}{1-dt}=1+\sum_{n\geq 1}d^{n-1}t^n.
\]
The latter equality means that $\dim\left((K\langle X_d\rangle^{C_d})^{(n)}\right)=d^{n-1}$, $n=1,2,\ldots$.

(ii) The statement follows immediately from Lemma \ref{generating function} because
\[
H(K\langle X_d\rangle^{C_d},t)=\frac{1-(d-1)t}{1-dt}=\frac{1}{1-g(t)}
\]
implies that $\displaystyle g(t)=\frac{t}{1-(d-1)t}$.
\end{proof}

We shall change the basis of the vector space ${\mathbb C}X_d$ and shall work in the free associative algebra
${\mathbb C}\langle Y_d\rangle={\mathbb C}\langle y_0,y_1,\ldots,y_{d-1}\rangle$,
where $C_d$ acts on the generators $y_0,y_1,\ldots,y_{d-1}$ by
\[
\rho(y_k)=\varepsilon^ky_k,\; k=0,1,\ldots,d-1,\; \varepsilon=\cos\left(\frac{2\pi}{d}\right)+i\sin\left(\frac{2\pi}{d}\right).
\]

Now we shall describe a vector space basis of the algebra ${\mathbb C}\langle Y_d\rangle^{C_d}$ and a set of its free generators.

\begin{theorem}\label{description of algebra of invariants}
{\rm (i)} The algebra ${\mathbb C}\langle Y_d\rangle^{C_d}$ has a basis which consists of all monomials
$y_{i_1}\cdots y_{i_n}$ such that $i_1+\cdots+i_n\equiv 0\text{ \rm (mod }d)$.

{\rm (ii)} The set of all monomials from {\rm (i)}
such that for all $m=1,2,\ldots,n-1$ the beginning monomial $y_{i_1}\cdots y_{i_m}$ satisfies $i_1+\cdots+i_m\not\equiv 0\text{ \rm (mod }d)$
forms a set of free generators of ${\mathbb C}\langle Y_d\rangle^{C_d}$.
\end{theorem}

\begin{proof}
(i) The generator $\rho$ of $C_d$ acts on the monomials of ${\mathbb C}\langle Y_d\rangle$ by
\[
\rho(y_{i_1}\cdots y_{i_n})=\varepsilon^{i_1+\cdots+i_n}y_{i_1}\cdots y_{i_n}
\]
and hence $y_{i_1}\cdots y_{i_n}\in{\mathbb C}\langle Y_d\rangle^{C_d}$ if and only if $i_1+\cdots+i_n\equiv 0\text{ \rm (mod }d)$.

(ii) The set $Z$ of the monomials from (i) forms a subsemigroup of the free noncommutative multiplicative semigroup $\langle Y_d\rangle$ generated by $Y_d$. By (i)
${\mathbb C}\langle Y_d\rangle^{C_d}$ is the semigroup algebra of the semigroup $Z$. Hence $Z$ is a free semigroup because the algebra ${\mathbb C}\langle Y_d\rangle^{C_d}$ is free.
If $y_{i_1}\cdots y_{i_n}\in{\mathbb C}\langle Y_d\rangle^{C_d}$ and for some beginning $y_{i_1}\cdots y_{i_m}$ we have
$i_1+\cdots+i_m\equiv 0\text{ \rm (mod }d)$, then a similar congruence $i_{m+1}+\cdots+i_n\equiv 0\text{ \rm (mod }d)$ holds for the ending
$y_{i_{m+1}}\cdots y_{i_n}$. Hence both $y_{i_1}\cdots y_{i_m}$ and $y_{i_{m+1}}\cdots y_{i_n}$ belong to ${\mathbb C}\langle Y_d\rangle^{C_d}$
and $y_{i_1}\cdots y_{i_n}=(y_{i_1}\cdots y_{i_n})(y_{i_{m+1}}\cdots y_{i_n})$ is not a free generator of ${\mathbb C}\langle Y_d\rangle^{C_d}$.
If $i_1+\cdots+i_m\not\equiv 0\text{ \rm (mod }d)$ for all $m=1,\ldots,n-1$, this means that $y_{i_1}\cdots y_{i_n}$ is a free generator of the semigroup $Z$
and the set of free generators of $Z$ is also a set of free generators of ${\mathbb C}\langle Y_d\rangle^{C_d}$.
\end{proof}

Below we give an algorithm how to construct a free generating set of the algebra ${\mathbb C}\langle Y_d\rangle^{C_d}$.

\begin{algorithm}\label{algorithm}
The following procedure allows to construct inductively a system of free generators of $K\langle Y_d\rangle^{C_d}$.

(i) Let $n=1$. There is only one monomial $y_0$ which satisfies the condition of Theorem \ref{description of algebra of invariants} (i).
We form the set $Z_1=\{y_0\}$ which consists of free generators of degree 1 and the set $U_1=\{y_1,\ldots,y_{d-1}\}$ of monomials of degree 1 which are not free generators.
Obviously $\vert Z_1\vert=1$ and $\vert U_1\vert=d-1$.

(ii) Let $n=2$. For each of the $d-1$ elements $y_i$ of $U_1$ there is exactly one element $y_j=y_{d-i}$ such that $y_iy_j\in {\mathbb C}\langle Y_d\rangle^{C_d}$.
We include these $y_iy_j$ in the set $Z_2$ of free generators of degree 2. Obviously $\vert Z_2\vert=d-1$. For each $y_i\in U_1$ there are $d-1$ variables $y_j$
such that $y_iy_j$ does not belong to ${\mathbb C}\langle Y_d\rangle^{C_d}$. We include these products $y_iy_j$ in the set $U_2$. Hence $\vert U_2\vert=(d-1)^2$.

(iii) Let $n>2$. By induction we have constructed the set $Z_{n-1}$ of monomials of degree $n-1$ which are free generators and the set $U_{n-1}$ of $(d-1)^{n-1}$ monomials
which are not in ${\mathbb C}\langle Y_d\rangle^{C_d}$ and do not have a beginning which is in ${\mathbb C}\langle Y_d\rangle^{C_d}$.
For each $y_{i_1}\cdots y_{i_{n-1}}\in U_{n-1}$ there is a single variable $y_j$ such that $y_{i_1}\cdots y_{i_{n-1}}y_j\in {\mathbb C}\langle Y_d\rangle^{C_d}$. We include these
$\vert U_{n-1}\vert=(d-1)^{n-1}$ monomials in the set $Z_n$ of generators of degree $n$, which agrees with Theorem \ref{Hilbert series and free generators} (ii).
The other $d-1$ monomials $y_{i_1}\cdots y_{i_{n-1}}y_k$, $k\not=j$,
have the property that neither they nor their beginnings belong to ${\mathbb C}\langle Y_d\rangle^{C_d}$. They form the set $U_n$ with $\vert U_n\vert=\vert U_{n-1}\vert(d-1)=(d-1)^n$.

In this way, step by step, we shall construct the free generating set of ${\mathbb C}\langle Y_d\rangle^{C_d}$.

(iv) For each of the $d^{n-1}$ monomials $y_{i_1}\cdots y_{n-1}\in {\mathbb C}\langle Y_d\rangle$ there is a unique $y_j$ such that $y_{i_1}\cdots y_{n-1}y_j\in {\mathbb C}\langle Y_d\rangle^{C_d}$
and this agrees with the equality $\dim\left((K\langle X_d\rangle^{C_d})^{(n)}\right)=d^{n-1}$ from Theorem \ref{Hilbert series and free generators} (i).
\end{algorithm}

\begin{remark}\label{different fields}
As in the commutative case, by Theorem \ref{Hilbert series and free generators} the dimension of $(K\langle X_d\rangle^{C_d})^{(n)}$
and the number of free generators of degree $n$ in $K\langle X_d\rangle^{C_d}$ do not depend on the field $K$.
Hence any basis of the $\mathbb Q$-vector space $\left({\mathbb Q}\langle X_d\rangle^{C_d}\right)^{(n)}$ is a basis also for $\left(K\langle X_d\rangle^{C_d}\right)^{(n)}$
Similarly, any basis of the factor space
\[
\left({\mathbb Q}\langle X_d\rangle^{C_d}\right)^{(n)}/\sum_{k=1}^{n-1}\left({\mathbb Q}\langle X_d\rangle^{C_d}\right)^{(k)}\left({\mathbb Q}\langle X_d\rangle^{C_d}\right)^{(n-k)}
\]
is a part of a free generating set of ${\mathbb Q}\langle X_d\rangle^{C_d}$ and hence also of $K\langle X_d\rangle^{C_d}$.
\end{remark}

\section{\textbf{The $S$-algebra} $(\C \langle Y_d \rangle^{C_d},\circ)$}
Till the end of the paper we consider the free algebra $K\langle X_d\rangle$ as an $S$-algebra,
equipping the homogeneous component $(K\langle X_d\rangle)^{(n)}$ of degree $n$ with the right action of the symmetric group $\text{\rm Sym}_n$
that permutes positions of the variables, namely
\[
(x_{i_1}\cdots x_{i_n})\circ \sigma=x_{i_{\sigma^{-1}(1)}}\cdots x_{i_{\sigma^{-1}(n)}},\;\sigma\in \text{\rm Sym}_n.
\]
As we have mentioned in the introduction, this action was introduced by Kharchenko \cite{Kh2} and investigated by Koryukin \cite{Kor}.
In this section we shall describe the generators of the $S$-algebra $(\C \langle Y_d \rangle^{C_d},\circ)$.

\begin{theorem}\label{cyclic S-invariants}
Let
\[
U=\{u_0=y_0,u_1,\ldots,u_k\},\;u_i=y_1^{n_{i,1}}\cdots y_{d-1}^{n_{i,d-1}},\;i=1,\ldots,k,
\]
be the set of generators of the algebra ${\mathbb C}[Y_d]^{C_d}$ from Proposition \ref{commutative case}.
Then the same monomials considered as elements of ${\mathbb C}\langle Y_d\rangle$ generate the algebra
$({\mathbb C}\langle Y_d\rangle^{C_d},\circ)$.
\end{theorem}

\begin{proof}
By Theorem \ref{description of algebra of invariants} (i) the algebra ${\mathbb C}\langle Y_d\rangle^{C_d}$ has a basis of all monomials
$y_{i_1}\cdots y_{i_n}$ such that $i_1+\cdots+i_n\equiv 0\text{ \rm (mod }d)$.
Hence it is sufficient to show that such monomials belong to the $S$-subalgebra of ${\mathbb C}\langle Y_d\rangle$ generated by the set $U$.
There is a $\sigma\in \text{\rm Sym}_n$ such that
\[
y_{i_1}\cdots y_{i_n}\circ\sigma=y_0^{n_0}y_1^{n_1}\cdots y_{d-1}^{n_{d-1}}
\]
and $n_1+2n_2+\cdots+(d-1)n_{d-1}\equiv 0\text{ \rm (mod }d)$. Working in ${\mathbb C}[Y_d]^{C_d}$,
\[
y_0^{n_0}y_1^{n_1}\cdots y_{d-1}^{n_{d-1}}=u_0^{n_0}u_1^{p_1}\cdots u_k^{p_k}
\]
for suitable $p_1,\ldots,p_k$. Going back to ${\mathbb C}\langle Y_d\rangle^{C_d}$ there is a $\tau\in \text{\rm Sym}_n$ such that
\[
y_0^{n_0}y_1^{n_1}\cdots y_{d-1}^{n_{d-1}}\circ\tau=u_0^{n_0}u_1^{p_1}\cdots u_k^{p_k}.
\]
Hence $y_{i_1}\cdots y_{i_n}\circ\sigma\tau=u_0^{n_0}u_1^{p_1}\cdots u_k^{p_k}$ and this shows that all monomials in the basis of
${\mathbb C}\langle Y_d\rangle^{C_d}$ belong to the $S$-subalgebra of ${\mathbb C}\langle Y_d\rangle^{C_d}$ generated by the set $U$.
\end{proof}

\begin{example}
It follows from Example \ref{invariants with diagonal action} that the following monomials generate the $S$-algebra
$({\mathbb C}\langle Y_d\rangle^{C_d},\circ)$ for $d=3,4,5$:

(i) $d=3$. The $S$-algebra $({\mathbb C}\langle Y_3\rangle^{C_3},\circ)$ is generated by $y_0,y_1y_2,y_1^3,y_2^3$.

(ii) $d=4$. The set $\{y_0, y_1y_3,y_2^2,y_1^2y_2,y_2y_3^2,y_1^4,y_3^4\}$ generates $({\mathbb C}\langle Y_4\rangle^{C_4},\circ)$.

(iii) $d=5$. The $S$-algebra ${\mathbb C}\langle Y_5\rangle^{C_5}$ is generated by
\[
y_0,y_1y_4,y_2y_3,y_1^2y_3,y_1y_2^2,y_1^3y_2,y_1y_3^3,y_2^3y_4,y_3y_4^3,y_1^5,y_2^5,y_3^5,y_4^5.
\]
\end{example}

\section{\textbf{The $S$-algebra} $(\C \langle X_3 \rangle^{C_3},\circ)$}
As in Example \ref{three and four commuting variables} it is interesting to find the generators of the $S$-algebra
$(K\langle X_d\rangle^{C_d},\circ)$ in terms of the canonical action of $C_d$ by shifting the variables.
The case $d=2$ is well understood and its investigation started in the paper by Wolf \cite{Wo}.

\begin{example}\label{three noncommutative variables}
We shall show that the $S$-algebra $(K\langle X_3\rangle^{C_3},\circ)$ is generated by the same cyclic sums as $K[X_3]^{C_3}$:
\[
v_1=x_1+x_2+x_3,\;v_2=x_1x_2+x_2x_3+x_3x_1,
\]
\[
v_{31}=x_1^2x_2+x_2^2x_3+x_3^2x_1,\;v_{32}=x_1x_2^2+x_2x_3^2+x_3x_1^2.
\]
The $S$-algebra $(K\langle Y_3\rangle^{C_3},\circ)$ is generated by $y_0,y_1y_2,y_1^3,y_2^3$, where
\[
y_0=x_1+x_2+x_3,\,y_1=x_1+\varepsilon^2x_2+\varepsilon x_3,\,y_2=x_1+\varepsilon x_2+\varepsilon^2x_3,\,\varepsilon=-\frac{1}{2}+i\frac{\sqrt{3}}{2}.
\]
Since $y_0=v_1$, it is sufficient to show that $y_1y_2,y_1^3,y_2^3$ belong to the $S$-subalgebra of $({\mathbb C}\langle X_3\rangle,\circ)$
generated by $v_1,v_2,v_{31},v_{32}$. Direct computations show that
\[
y_1y_2=v_1^2+(\varepsilon-1)v_2+(\varepsilon^2-1)v_2\circ(12),
\]
\[
y_1^3=v_1^3+(\varepsilon^2-1)v_{31}\circ(1+(23)+(13))+(\varepsilon-1)v_{32}\circ(1+(23)+(13)),
\]
\[
y_2^3=v_1^3+(\varepsilon-1)v_{31}\circ(1+(23)+(13))+(\varepsilon^2-1)v_{32}\circ(1+(23)+(13)).
\]
Direct computations also show that this is system is minimal. For example, we cannot express $y_1^3$ in terms of the operation $\circ$ by $v_1,v_2,v_{31}$.
\end{example}

When we consider the graded free subalgebras of $K\langle X_d\rangle$, the number of homogeneous free generators of a given degree is the same for all systems of free generators.
We do not know whether a similar fact holds for the number of generators of the $S$-subalgebras $(F,\circ)$ of $(K\langle X_d\rangle,\circ)$.
Since the action of Kharchenko-Koryukin commutes with the linear change of the basis of the vector space $KX_d$,
what is obvious, that a generating system of $(F,\circ)$ does not depend on the basis of $KX_d$.

\begin{conjecture}\label{conjecture}
Every two minimal generating systems of an $S$-subalgebra of $(K\langle X_d\rangle,\circ)$ have the same number of elements of a given degree.
\end{conjecture}

Example \ref{three noncommutative variables} confirms Conjecture \ref{conjecture} for the special case of $({\mathbb C}\langle X_3\rangle^{C_3},\circ)$
and two explicitly given generating systems.

\section*{Acknowledgements}

The first named author was partially supported  by the European Union-NextGenerationEU,
through the National Recovery and Resilience Plan of the Republic of Bulgaria, Project № BG-RRP-2.004-0008-C01.


\end{document}